\documentclass[12pt, letterpaper,reqno]{amsart}
\usepackage{amsthm,amssymb,mathrsfs}
\usepackage{amsmath,amscd}
\usepackage{graphicx} 
\usepackage{xcolor}
\usepackage{tikz}
\usepackage{fourier}
\usepackage[breaklinks,colorlinks,citecolor=blue,linkcolor=blue,
 urlcolor=teal]{hyperref} 

\theoremstyle{plain}
\newtheorem{theorem}[equation]{Theorem}
\newtheorem{proposition}[equation]{Proposition}
\newtheorem{lemma}[equation]{Lemma}

\theoremstyle{definition}
\newtheorem{definition}[equation]{Definition}
\newtheorem{remark}[equation]{Remark}

\newtheorem*{ack}{Acknowledgments}

\newcommand{\co}{\colon \thinspace}

\newcommand{\Z}{{\mathbb Z}}

\newcommand{\abs}[1]{\lvert {#1} \rvert}

\renewcommand{\leq}{\leqslant}
\renewcommand{\geq}{\geqslant}
\renewcommand{\epsilon}{\varepsilon}
\renewcommand{\phi}{\varphi}

\numberwithin{equation}{section}

\makeatletter
  \def\tagform@#1{\maketag@@@{%
   \textbf{(\ignorespaces#1\unskip\@@italiccorr)}}}%
   \renewcommand{\eqref}[1]{\textup{\maketag@@@{(\ignorespaces%
        {\ref{#1}}\unskip\@@italiccorr)}}}
\makeatother

\begin{document}

\title[Finite-index accessibility]{Finite-index accessibility of groups}

\author{Max Forester}
\address{Mathematics Department\\
        University of Oklahoma\\
        Norman, OK 73019\\
        USA}
\author{Anthony Martino}
\address{1640 Redfield St.\\
        La Crosse, WI 54601\\
        USA}

\email{mf@ou.edu, martino.m.tony@gmail.com}


\begin{abstract}
We prove an accessibility theorem for finite-index splittings of
groups. Given a finitely presented group $G$ there is a number $n(G)$
such that, for every reduced locally finite $G$--tree $T$ with finitely
generated stabilizers, $T/G$ has at most $n(G)$ vertices and edges. 
We also show that deformation spaces of locally finite trees (with
finitely generated stabilizers) are maximal in the partial ordering of
domination of $G$--trees. 
\end{abstract}

\maketitle

\thispagestyle{empty}

\section{Introduction}

Given a group $G$, one may wish to understand the various ways in
which $G$ decomposes as the fundamental group of a graph of groups.
For simplicity, we shall call a graph of groups decomposition a
\emph{splitting} of $G$. There are several theorems in the literature,
known as \emph{accessibility} theorems, which state that all splittings
of $G$ of a particular kind have underlying graphs of bounded
complexity. That is, there is a bound on the total number of vertices
and edges. 

An early result of this kind concerns free products. Recall that the
\emph{rank} of a group is the smallest size of a generating set. By
Grushko's theorem, rank is additive under free products. It follows
easily that for a finitely generated group $G$, there is a bound on
the complexity of any reduced splitting of $G$ having trivial edge
groups. 

Dunwoody's accessibility theorem \cite{dunwoody} provides a complexity
bound on all reduced splittings of $G$ having finite edge groups, if
$G$ is finitely presented. A fundamental structure theorem then
follows from this result: every such group admits a decomposition
whose edge groups are finite and whose vertex groups each have at most
one end. Dunwoody accessibility does not hold for all finitely
generated groups \cite{dunwoody2}, though Linnell had shown earlier
that it does if one assumes a bound on the orders of the finite subgroups
of $G$ \cite{linnell}.

Bestvina and Feighn \cite{bestvinafeighn,bestvinafeighn2} generalized
Dunwoody's theorem by bounding the complexity of all reduced
splittings of $G$ having small edge groups. Here, a group is
\emph{small} if it does not admit an irreducible action on a
simplicial tree. (Note that every group with no non-abelian free
subgroup is small.) Bestvina and Feighn also gave examples showing
that once one allows the free group of rank 2 as an edge group, then
no such bounds can exist. 

Sela's acylindrical accessibility theorem \cite{sela} generalizes the
free product case in a different way. A graph of groups is called
\emph{$k$--acylindrical} if its associated Bass-Serre tree has the
property that the stabilizer of every segment of length greater than
$k$ is trivial. Sela proved that for every finitely generated group
$G$ and $k\geq 1$, there is a bound on the complexity of all reduced
$k$--acylindrical splittings of $G$. 

Sela's result has been generalized to $(k,C)$--acylindrical
splittings of groups by Delzant and Weidmann \cite{delzant,weidmann},
where a splitting is called \emph{$(k,C)$--acylindrical} if every
segment of length greater than $k$ has stabilizer of cardinality at
most $C$. 

Finally there is a notion of \emph{strong accessibility}, which
concerns finiteness of hierarchies of splittings of a given type. See
Louder and Touikan \cite{loudertouikan} for results on hierarchies of
splittings over slender groups of finitely presented groups. 

\medskip

In the current paper we prove an accessibility theorem for
\emph{finite-index splitting}s of groups. These are graph of groups
decompositions in which every edge group includes as a finite-index
subgroup of its neighboring vertex groups. The Bass-Serre trees of
such splittings are exactly the $G$--trees that are locally
finite.%
\footnote{Here we are only considering graphs of groups with finite
  underlying graphs.}

Each of the earlier results has a hypothesis limiting the edge
groups or segment stabilizers in size or complexity. In our theorem
the limitation works in the other direction, with edge groups being as
large as possible relative to the vertex groups. 

Recall from \cite{stallings} that a group $G$ is \emph{almost finitely
presented} if there is a connected simplicial complex $K$ with
$H^1(K;\Z_2) = 0$ on which $G$ acts freely and cocompactly. Every
finitely presented group is almost finitely presented, but not
conversely (see \cite{bieristrebel} for a counterexample). 

\begin{theorem}\label{mainthm} 
Let $G$ be an almost finitely presented group. Then there is an
integer $n(G)$ such that the following holds: if $T$ is
a reduced locally finite $G$--tree with finitely generated 
stabilizers, then $T/G$ has at most $n(G)$ vertices and edges. 
\end{theorem}

Note that in any locally finite $G$--tree, all vertex and edge
stabilizers are commensurable. Hence either they are all finitely
generated, or none of them are. 

\begin{remark}
The assumption of finitely generated stabilizers cannot be
dropped. At the end of the paper we give examples of finite-index
splittings of the free group $F_2$, of arbitrarily large complexity, 
with all vertex and edge groups free of infinite rank. 
\end{remark}

Our main technical result leading to Theorem \ref{mainthm} is a
statement about deformation spaces of locally finite trees. Let $X$
and $Y$ be cocompact $G$--trees. We say that $X$ \emph{dominates} $Y$
if there is a simplicial $G$--map $X \to Y$ (after possibly
subdividing $X$). This occurs if and only
if every elliptic subgroup for $X$ is elliptic for $Y$. The trees are
in the same \emph{deformation space} if they dominate each other, or
equivalently, have the same elliptic subgroups. There is then an
induced partial ordering of domination of deformation spaces. 
The next result says that deformation spaces of non-trivial locally
finite trees (with finitely generated stabilizers) are maximal in this
partial ordering. 

\begin{theorem}\label{mainthm2} 
Let $Y$ be a non-trivial locally finite $G$--tree with finitely
generated stabilizers. If $f \co X \to Y$ is any surjective
simplicial map of $G$--trees, with $X$ cocompact, then $X$ and $Y$ are
in the same deformation space. 
\end{theorem}

If $X$ is minimal then this implies that it must also be locally
finite. 
Like
Theorem \ref{mainthm}, this result fails if one does not assume $Y$ to
have finitely generated stabilizers. Consider any HNN extension $G = A 
\ast_C$ whose Bass--Serre tree $X$ is not a line. Let $H\subset G$ be 
the subgroup generated by all conjugates of $A$.  This subgroup is not
finitely generated and is not elliptic for $X$.  However, there is a
morphism $X \to Y$ onto a linear $G$--tree whose vertex and edge
stabilizers are all $H$. These trees are in different deformation
spaces. 

The assumption of non-triviality is also essential, as every $G$--tree
admits a simplicial $G$--map to a point. 

\begin{ack}
The first-named author was partially supported by Simons Foundation
award \#638026. We thank the referee for several helpful suggestions. 
\end{ack}

\section{Preliminaries}

For more detail on the background material given here, see
\cite{basscovering,defrigid,bestvinafeighn}. 
Our convention for graphs and trees is that every
edge $e$ is oriented, with initial vertex $\partial_0(e)$ and terminal
vertex $\partial_1(e)$. Then $\overline{e}$ denotes the same geometric
edge with
the opposite orientation, so that $\partial_0(\overline{e})
= \partial_1(e)$ and $\partial_1(\overline{e}) = \partial_0(e)$. 

By a \emph{$G$--tree} we mean a simplicial tree together with a
$G$--action by simplicial automorphisms, with no inversions.
A $G$--tree is \emph{trivial} it it has a global fixed
point; otherwise, it is \emph{non-trivial}. 
Let $X$ be a $G$--tree. An element $g\in G$ is \emph{elliptic} if it
fixes a vertex and \emph{hyperbolic} otherwise. If $g$ is hyperbolic,
then there is a $g$--invariant line $X_g \subset X$ called the
\emph{axis} of $g$ on which $g$ acts as a translation. 

An \emph{end} of $X$ is an equivalence class of rays, where two rays
are considered equivalent if their intersection is a ray. The space of
ends of $X$ is denoted by $\partial X$. A subtree \emph{contains} the
end $\varepsilon \in \partial X$ if it contains a ray representing
$\varepsilon$. 

\begin{lemma}[\cite{tits}, (3.4)] 
Let $X$ be a $G$--tree and $H \subset G$ a subgroup such that every
element of $H$ is elliptic. Then either: 
\begin{enumerate}
\item\label{e1} $H$ fixes a vertex of $X$, or 
\item\label{e2} $H$ has no fixed vertex, but there is a unique end
  $\varepsilon_H \in \partial X$ fixed by $H$.  
\end{enumerate}
\end{lemma}

\noindent
In case \eqref{e1}, we say that $H$ is \emph{elliptic}. In case
\eqref{e2}, we say that $H$ is \emph{elliptic at infinity}. 

If $X$ is a $G$--tree, a \emph{minimal subtree} is a subtree
that is $G$--invariant and does not properly contain any
smaller $G$--invariant subtree. If $G$ is elliptic, then any
vertex fixed by $G$ is a minimal subtree. If $G$ contains a hyperbolic 
element, then a minimal subtree $X_G$ exists and is unique: it is
equal to the union of the axes of hyperbolic elements. The third
possibility is that $G$ is elliptic at infinity; then there is no
minimal subtree. Instead, there
is a descending chain of $G$--invariant subtrees, with empty
intersection, all containing the fixed end $\varepsilon_G$
(for instance, take a sequence of horoballs at $\varepsilon_G$). 

\begin{lemma}\label{subtree}
Let $X'$ be a $G$--invariant subtree of the $G$--tree $X$. 
\begin{enumerate}
\item \label{i1} $X$ and $X'$ have the same elliptic subgroups.
\item \label{i2} If\/ $G$ is elliptic at infinity for $X$ then $X'$
  contains the fixed end $\epsilon_G$. 
\item \label{i3} If\/ $G$ has a hyperbolic element then $X'$
  contains the minimal subtree $X_G$. 
\end{enumerate}
\end{lemma}

\begin{proof}
First note that $X$ and $X'$ have the same elliptic and hyperbolic
elements; each $g$ has either a fixed vertex or an axis in $X'$, and
then has the same in $X$. Item
\eqref{i3} follows immediately since $X'$ must contain every
axis. Next, a subgroup cannot fix both a vertex of $X$ and an end of
$X'$ without also fixing a vertex of $X'$, so \eqref{i1}
holds. Finally, if $G$ is elliptic at infinity for $X$, it is also
elliptic at infinity for $X'$, by \eqref{i1}. So $G$ fixes a unique
end $\varepsilon$ of $X'$. Uniqueness of $\varepsilon_G$ for $X$
implies that $\varepsilon_G = \varepsilon$. 
\end{proof}

Recall that two subgroups $H, H'$ of a group $G$ are 
\emph{commensurable} if $H \cap H'$ has finite index in both $H$ and
$H'$. In any locally finite $G$--tree, 
stabilizers of vertices and edges are all commensurable to each other. 

\begin{lemma}\label{comm}
Let $X$ be a $G$--tree and let $H, H'$ be commensurable subgroups of
$G$. 
\begin{enumerate}
\item\label{comm1} If $H$ is elliptic then so is $H'$. 
\item\label{comm2} If $H$ is elliptic at infinity, then so is $H'$ and
  $\varepsilon_H = \varepsilon_{H'}$. 
\item\label{comm3} If $H$ contains a hyperbolic element, then so does
  $H'$ and $X_H = X_{H'}$. 
\end{enumerate}
\end{lemma}

\begin{proof}
Item \eqref{comm1} is immediate from Example 6.3.4 of
\cite{serre}. Item \eqref{comm3} is Corollary 7.7 of
\cite{basscovering}. For \eqref{comm2}, $H'$ being elliptic at
infinity follows from \eqref{comm1} and \eqref{comm3}. Further, if
$\varepsilon_H \not= \varepsilon_{H'}$, then $H \cap H'$ fixes
pointwise the geodesic line joining these two ends. But then $H \cap
H'$ is elliptic, contradicting \eqref{comm1}. 
\end{proof}

\begin{definition}[Collapse maps]
Let $f \co X \to Y$ be an equivariant simplicial map of $G$--trees. It
is called a \emph{collapse map} if it is surjective and $f^{-1}(v)$ is
connected for every vertex $v \in V(Y)$. In terms of graphs of groups,
this means that $Y/G$ is obtained from $X/G$ by collapsing the
components of a subgraph of $X/G$ to vertices. The vertex groups
associated to these new vertices are the fundamental groups of the
collapsed subgraphs of groups. 
\end{definition}

\begin{definition}[Morphisms]
An equivariant map $f \co X \to Y$ is a \emph{morphism} if it takes
vertices to vertices and edges to edges (and respects the relation of
incidence between vertices and edges). Equivalently, it is a
simplicial map such that no edge is collapsed to a vertex. 
\end{definition}

\begin{definition}[Elementary moves]
Let $X$ be a $G$--tree and $e$ an edge of $X$ with endpoints in
distinct $G$--orbits, such that $G_{\partial_0 (e)} = G_e$. We obtain
a new $G$--tree by collapsing each edge in the orbit of $e$ to a
vertex. This operation is called an \emph{elementary collapse}
move. It is a special case of a collapse map from $X$ to the resulting
tree. What makes it different from a generic collapse map is that the
elliptic subgroups do not change.  The reverse move is called an
\emph{elementary expansion}. 
\end{definition}

\begin{definition}[Deformations]
An \emph{elementary deformation} between $G$--trees $X$ and $Y$ is a
finite sequence of elementary collapse and expansion moves taking $X$
to $Y$. When this occurs, we say that they are in the same
\emph{deformation space}. Elementary deformations preserve the
elliptic subgroups of $G$, and moreover, cocompact $G$--trees are in
the same deformation space if and only if they define the same
elliptic subgroups \cite[Theorem 1.1]{defrigid}. 
\end{definition}

\begin{definition}[Reduced trees]
A $G$--tree is \emph{reduced} if it admits no elementary collapse
moves. Any cocompact $G$--tree can be made reduced by a sequence of
such moves. Note that reduced trees are minimal; one easily verifies
that any edge outside of a $G$--invariant subtree admits an elementary
collapse. 

Warning: our definition of ``reduced'' is not the same as the one
used in \cite{bestvinafeighn}. A $G$--tree is \emph{BF-reduced}
(i.e. reduced in the sense of Bestvina-Feighn) if it is minimal and,
whenever $G_e = G_{\partial_0(e)}$, either $e$ maps to a loop in the
quotient graph or the image of $\partial_0(e)$ in
the quotient graph has valence at least 3. Note that every reduced
$G$--tree is BF-reduced. 

Theorem \ref{mainthm} holds both for reduced trees and for BF-reduced
trees; see Remark \ref{reducedrmk}. 
\end{definition}

\begin{definition}[Folds]
Let $X$ be a $G$--tree and suppose $e$, $e'$ are edges with
$\partial_0 (e) = \partial_0 (e') = v$. Let $u = \partial_1 (e)$ and $u'
= \partial_1 (e')$. Define an equivalence relation $\sim$ on $X$ to be
the smallest equivalence relation satisfying: 
\begin{itemize}
\item $u \sim u'$, $e \sim e'$, and $\overline{e} \sim \overline{e}'$ 
\item $g(x) \sim g(y)$ whenever $x \sim y$ and $g\in G$. 
\end{itemize}
The resulting quotient graph is a simplicial tree with a
$G$--action. There could be inversions, in which case we subdivide the
inverted edge orbit to obtain a $G$--tree. The map $X \to X/\!\!\sim$
is called a \emph{fold}.

It is called a
\emph{type A fold} if $v \not\in Gu \cup Gu'$.
Note that, by minimality of the relation $\sim$, all identifications
among vertices occur within $Gu \cup Gu'$. Hence, in a type A fold, no
identifications occur between vertices of the orbit $Gv$.

\end{definition}
We will make use of the following observation from
\cite[Section~2]{bestvinafeighn}: 

\begin{lemma}\label{typeA}
  Every fold either is of type A, or can be achieved using
  finitely many of the following moves: 
subdivision, type A folds, and the reverse of subdivision. 
\end{lemma}

A key tool used in the proof of Theorem \ref{mainthm} is the 
Dunwoody resolution lemma. First we need a definition, from Dunwoody
and Fenn \cite{dunwoodyfenn}. If $X$ is a 
$G$--tree, a vertex $x$ is \emph{inessential} if it is the initial
endpoint of exactly two edges $e_1$ and $e_2$ and $G_{e_1} = G_x =
G_{e_2}$. Otherwise $x$ is called \emph{essential}. Note that if $x$
is inessential and $y$ is an essential vertex closest to $x$, then
$G_x \subseteq G_y$. Hence every elliptic subgroup fixes an essential
vertex (unless there are none, meaning that $X$ is a line and $G$ acts
by translations). 

\begin{lemma}[Dunwoody resolution lemma]\label{dunwoodylemma}
Let $G$ be an almost finitely presented group. Then there is an integer
$\delta(G)$ such that the following holds: if\/ $T$ is any $G$--tree,
there is a cocompact $G$--tree $T'$ and a simplicial $G$--map $\alpha
\co T' \to T$ such that $T'$ has at most $\delta(G)$ orbits of
essential vertices. 
\end{lemma}

The lemma is Theorem 1.6 of \cite{dunwoodyfenn}. Cocompactness of $T'$
is not explicitly mentioned, but it is evident from the proof. The
number $\delta(G)$ comes from the simplicial complexes $K$ witnessing
the almost finite presentability of $G$. If $L = K/G$ has $\ell_0$
vertices and $\ell_2$ $2$--simplices, then $\delta(G) \leq 2 \dim
H^1(L;\Z_2) + \ell_0 + \ell_2$. 

\begin{lemma}\label{lfsubtree} 
Let $X$ and $Y$ be $G$--trees in the same deformation space and
suppose that $X$ contains a locally finite subtree that is
$G$--invariant. Then $Y$ also contains a $G$--invariant locally finite
subtree. 
\end{lemma}

Note that local finiteness itself is not preserved by elementary
deformations. An expansion move outside of a minimal subtree can create
infinite-valence vertices. 

\begin{proof}
It suffices to consider a single elementary collapse move $q\co X \to
Y$ along the edge $e \in E(X)$ with $G_{\partial_0 (e)} = G_e$. First
suppose that $X$ has a $G$--invariant locally finite subtree $X'$. Let
$Y' = q(X')$, a $G$--invariant subtree of $Y$. If $e \not\in X'$ then
$X'$ maps isomorphically to $Y'$, and so $Y'$ is locally
finite. Otherwise, the restriction $q\co X' \to Y'$ is an elementary
collapse. In the proof of \cite[Theorem 7.3]{defrigid} it was observed
that $q$ is a $(3, 2/3)$--quasi-isometry, and that
$\frac{1}{3}(d(x,x') - 2) \leq d(q(x),q(x'))$ for all $x,x'\in X'$. It
follows that the pre-image of a ball of radius $1$ is contained in a
ball of radius $5$ in $X'$. The latter is finite, and so every ball of
radius $1$ in $Y'$ is finite. 

Next consider the same collapse move $q\co X \to Y$ but suppose that
$Y$ contains a locally finite $G$--invariant subtree $Y'$. We wish to
find the same in $X$. Let $Z \subset X$ be the $G$--invariant subgraph
whose edges are $\{ e' \in E(X) - (Ge \cup G\overline{e}) \mid q(e')
\in E(Y')\}$. If $Z$ is connected then let $X' = Z$; it maps
isomorphically to $Y'$ by $q$ and hence is locally finite. 

Otherwise, $\partial_0(e)$ is in $Z$. We define $X'$ to be $Z \cup (Ge
\cup G\overline{e}) = q^{-1}(Y')$. For any vertex $v\in V(X')$
consider the ball $B_v(1)$ of radius $1$ at $v$. Edges of $Z$ in
$B_v(1)$ map injectively to $Y'$, so there are finitely many. It
remains to bound the number of edges of $Ge \cup G\overline{e}$ in
$B_v(1)$.  Note that each component of $Ge \cup G\overline{e}$ is a
cone on some subset $S \subset G \partial_0(e)$ (with cone point in $G
\partial_1(e)$). Each vertex of $S$ is incident to an edge of $Z$ and 
these edges are all distinct. Hence $S$ is finite, because collapsing
$e$ results in a locally finite tree. Hence $Ge \cup G\overline{e}$ is
locally finite, and therefore $B_v(1)$ is finite. 
\end{proof}

\section{Main results}

Here we give two key propositions and then use them to prove Theorems
\ref{mainthm} and \ref{mainthm2}. In order for the main argument to
work, these propositions are stated for a slightly wider class of
trees than just locally finite trees. 

\begin{proposition}\label{collapse-def}
Let $Y$ be a $G$--tree that contains a non-trivial locally finite
$G$--invariant subtree $Y'$ (possibly equal to $Y$). If $q \co X \to
Y$ is any collapse map then $X$ and $Y$ have the same elliptic
subgroups. 
\end{proposition}

\begin{proof}
Note that $Y$ and $Y'$ have the same elliptic subgroups, by Lemma
\ref{subtree}\eqref{i1}. 
First we show that $q$ preserves hyperbolicity of elements of
$G$. Suppose $g$ is hyperbolic in $X$ and elliptic in $Y$, with
fixed vertex $y\in Y$. We may assume that $y \in Y'$. Let $H =
G_y$ and consider the $H$--action on $X$. Since $q$ is a collapse map,
$q^{-1}(y)$ is a subtree of $X$, and it is $H$--invariant. The hyperbolic
element $g$ is in $H$, and so $X_H$ is non-empty and unique, and is
contained in $q^{-1}(y)$. 

Now pick any vertex $y' \in Y'$ distinct from $y$, and let $H' =
G_{y'}$. Since $Y'$ is locally finite, the subgroups $H$ and
$H'$ are commensurable. By Lemma \ref{comm}\eqref{comm3}, $H'$
contains a hyperbolic element (so that $X_{H'}$ is non-empty and
unique) and $X_{H'} = X_H$. However, $X_{H'}$ is contained in the
$H'$--invariant subtree $q^{-1}(y')$ which is disjoint from
$q^{-1}(y)$. Thus we have a contradiction. 

Next we show that $X$ and $Y$ have the same elliptic subgroups. 
Let $y\in Y'$ be any vertex and let $H = G_y$. If $H$ is not elliptic
in $X$, it must be elliptic at infinity since $X$ and $Y$ have the
same elliptic elements (by the previous paragraphs). Its fixed end
$\varepsilon_H$ is an end of the $H$--invariant subtree
$q^{-1}(y)$. Now let $y'$ be any other vertex of $Y'$ and let $H' =
G_{y'}$. By Lemma \ref{comm}\eqref{comm2}, $H'$ is elliptic at
infinity and $\varepsilon_{H'} = \varepsilon_H$. But
$\varepsilon_{H'}$ is an end of the $H'$--invariant subtree
$q^{-1}(y')$ which is disjoint from $q^{-1}(y)$. We have a
contradiction, since two disjoint subtrees cannot have an end in
common. Thus $H$ is elliptic. 
\end{proof}

\begin{proposition}\label{fold-def}
Let $Y$ be a $G$--tree containing a non-trivial locally finite
$G$--invariant subtree $Y'$. If $f \co X \to Y$ is any fold then $X$
and $Y$ are in the same deformation space. 
\end{proposition}

\begin{proof}
By Lemma \ref{typeA}, any fold that is not of type A
can be achieved using subdivision, type A folds, and the
reverse of subdivision. Since subdivision preserves deformation
spaces, we may assume the fold is of type A.  Also, by Proposition
3.16 of \cite{defrigid}, it suffices to show that $f$ preserves
hyperbolicity of elements of $G$. 

Suppose the fold happens at $e, e' \in E(X)$ where $\partial_0 (e) 
= \partial _0 (e') = v$ and $\partial_1 (e) = u$, $\partial_1 (e') = 
u'$. Being of type A means that $v \not\in Gu \cup Gu'$, and therefore 
the orbit $Gv$ maps injectively to $Y$ under the fold. 

Now suppose that $g$ is hyperbolic in $X$ and elliptic in $Y$, with
fixed vertex $y\in Y$. We may assume that $y \in Y'$. The set $\{g^i
v\}_{i \in \Z}$ maps by $f$ to an infinite set of vertices that are
equidistant from $y$. In a locally finite tree, every bounded subtree
is finite, and so these vertices cannot lie in $Y'$.  Thus, $f(v) \in
Y - Y'$. It follows that $f(e) = f(e') \in Y - Y'$ as well. 

Note that $f$ is surjective, and since $Y'$ is not a point, it must
contain an edge. Hence there is an edge $e'' \in E(X)$ such that
$f(e'') \in Y'$. The orbits $Ge$ and $Ge'$ both land outside of $Y'$,
so $e'' \not\in Ge \cup Ge'$. Hence $Ge''$ maps injectively to
$Y$. But now the set $\{ g^i e''\}_{i \in \Z}$ maps by $f$ to an
infinite set of edges of $Y'$ that lie in a bounded neighborhood of
$y$, a contradiction. 
\end{proof}

\begin{proof}[Proof of Theorem \ref{mainthm2}] 
Let $f \co X \to Y$ be as in the statement of the theorem. 
Any simplicial map factors as a collapse map followed by a
morphism. Being surjective, the morphism factors as a finite
composition of folds, by \cite[Section 2]{bestvinafeighn}. This step
uses the assumptions that $X$ is cocompact and $Y$ has finitely
generated edge stabilizers. Now we have $f$ represented as 
\[ X \to Y_0 \to Y_1 \to \dotsm \to Y_n = Y\]
with $X \to Y_0$ a collapse map and each $Y_i \to Y_{i+1}$ a
fold. Note that none of the trees has a fixed point, since $Y$ does
not. 

Applying Proposition \ref{fold-def} to the last fold, we find that
$Y_{n-1}$ is related to $Y_n$ by a deformation, and so by Lemma
\ref{lfsubtree} it contains a $G$--invariant locally finite
subtree (which is not a point). Proceeding from right to left along
the sequence of folds, we infer 
the same properties for every $Y_i$. Now Proposition 
\ref{collapse-def} says that $X$ and $Y_0$ have the same elliptic
subgroups. Since $X$ and $Y_0$ are both cocompact, they are in the same
deformation space (along with $Y$). 
\end{proof}

\begin{proof}[Proof of Theorem \ref{mainthm}] 
Let $T$ be a reduced locally finite $G$--tree with finitely generated
stabilizers. It suffices to bound the number of vertices of
$T/G$, since there will then be at most $\abs{V(T/G)} - 1 +
\beta_1(G)$ edges.  We may assume $T$ is not a point or a line, since
$\abs{V(T/G)} \leq 2$ in those cases. 

By Lemma \ref{dunwoodylemma} there is a cocompact $G$--tree 
$T'$ and a simplicial $G$--map $\alpha \co T' \to T$ such that $T'$
has at most $\delta(G)$ orbits of essential vertices. The map is
surjective because $T$ is minimal. By Theorem \ref{mainthm2}, $T'$ and
$T$ are in the same deformation space. 

Now let $v$ be any vertex of $T$. The elliptic subgroup $G_v$ fixes an
essential vertex $w$ of $T'$. Either $\alpha(w) = v$ or $G_v$ fixes
the path from $v$ to $\alpha(w)$. In the latter case, the first edge
$e$ along this path satisfies $G_e = G_v$. Since $T$ is reduced, $e$
maps to a loop in $T/G$. The number of such edge orbits is bounded by
$\beta_1(G)$. In particular, the number of vertex orbits that are not
images of essential vertex orbits of $T'$ is bounded by
$\beta_1(G)$. Hence the total number of vertex orbits of $T$ is
bounded by $\delta(G) + \beta_1(G)$. 
\end{proof}

\begin{remark}\label{reducedrmk}
Theorem \ref{mainthm} remains true even if $T$ is only BF-reduced,
though the proof is considerably more involved. Fortunately, the
arguments of \cite{bestvinafeighn} can be applied verbatim. The proof
proceeds as above until one has the $G$--map $\alpha\co T'
\to T$ with both trees in the same deformation space. Then, the
proof of the main result of \cite[Section 4 ``The elliptic
  case'']{bestvinafeighn} applies. This is so because every edge
stabilizer of $T$ is elliptic in $T'$. That argument shows that 
$|V(T/G)| \leq  4\delta(G) + 9\beta_1(G) - 5$. 
\end{remark}

\section{Examples}

If one drops the assumption that stablilizers are finitely
generated, then Theorem \ref{mainthm} becomes false. Our examples are
based on $2$--generator generalized Baumslag--Solitar groups.  These
are discussed in \cite{levitt-quotients}.  They fall into two
families, one of which we describe here. Recall that a
\emph{generalized Baumslag--Solitar group} is a group that admits a
graph of groups decomposition in which every vertex and edge group is
infinite cyclic. The figure below describes such a graph of
groups. Each edge-to-vertex morphism $\Z \to \Z$ is multiplication by
a non-zero integer $q_i$ or $r_i$.
\[ 
\begin{tikzpicture} [x=1.1mm,y=1.1mm] 
\small
\filldraw[fill=black,thick] (5,5) circle (.6mm);
\filldraw[fill=black,thick] (23,5) circle (.6mm);
\filldraw[fill=black,thick] (41,5) circle (.6mm);
\filldraw[fill=black,thick] (69,5) circle (.6mm);
\filldraw[fill=black,thick] (87,5) circle (.6mm);

\draw[very thick] (5,5) -- (50,5);
\draw[very thick] (60,5) -- (87,5); 
\draw (55.3,5.03) node{$\dotsc$}; 

\draw (8,7.5) node {$q_0$};
\draw (20,7.5) node {$r_1$};
\draw (26,7.5) node {$q_1$};
\draw (38,7.5) node {$r_2$};
\draw (44,7.5) node {$q_2$};
\draw (64.5,7.5) node {$r_{k-1}$};
\draw (74,7.5) node {$q_{k-1}$};
\draw (84,7.5) node {$r_k$};
\end{tikzpicture}
\]
We shall assume the graph of groups is reduced, which means that
$r_i, q_j \not= \pm 1$ for all $i,j$. Then the group is $2$--generated
if and only if $r_i$ and $q_j$ are relatively prime whenever $1 \leq
i \leq j \leq k-1$; see the discussion on page 9 of 
\cite{levitt-quotients}, which depends on \cite[Theorem
1.1]{levitt-plateaux}. (One direction is Lemma 2.5 and the converse 
is an easy exercise.) The two generators are the generators of the
vertex groups at the terminal vertices.  

Let $G$ be the group shown and assume it is $2$--generated. The
Bass--Serre tree $X$ is reduced and locally finite, with vertices of
valence $|q_0|$, $|r_i| + |q_i|$, and $|r_k|$. The surjection $F_2 \to
G$ defines an action of $F_2$ on $X$ with the same quotient
graph. Hence edge stabilizers have index $\abs{q_i}$ or $\abs{r_i}$
in their neighboring vertex stabilizers, and $X$ is reduced as an
$F_2$--tree. These stabilizers are all free groups of infinite rank,
by a theorem of Bieri on cohomological dimension in this setting
\cite[Section 6]{bieri}. Namely, if these stabilizers were finite rank
free groups, then Bieri's theorem applies and says they each have
cohomological dimension zero. But then the quotient graph of groups
has trivial fundamental group, a contradiction. 

To conclude, say by taking $q_i = 2$ and $r_i = 3$ for all $i$, we
have a reduced locally finite $F_2$--tree with $k+1$ vertex orbits,
for arbitrary $k$. One further adjustment, setting $q_0 = r_k = 5$,
results in reduced $F_2$--actions on the \emph{same}
tree, the regular tree of valence 5, of arbitrarily large complexity. 

\bibliographystyle{amsalpha}
\bibliography{bib}

\end{document}